\documentclass{article}
\usepackage[utf8]{inputenc}
\usepackage{wrapfig}

\usepackage[utf8]{inputenc}

\usepackage{mathtools,amsmath,amssymb,amsfonts,amsthm,upgreek}
\usepackage{mathalfa}
\usepackage{microtype}
\usepackage{tikz,tikz-cd}
\usetikzlibrary{arrows}
\usetikzlibrary{decorations.markings}
\tikzset{degil/.style={line width=0.5pt,double distance=5pt,
        decoration={markings,
        mark= at position 0.5 with {
              \node[transform shape] (tempnode) {$\backslash\backslash$};
              }
          },
          postaction={decorate}
}
}

\tikzset{
commutative diagrams/.cd,
arrow style=tikz,
diagrams={>=open triangle 45, line width=0.5pt}}

\usetikzlibrary{calc}
\usepackage{color}
\usepackage{url}

\usepackage[backgroundcolor=white,linecolor=yellow,bordercolor=yellow,textcolor=black,textsize=small]{todonotes}

\graphicspath{{./figures/}}

\usepackage[english]{babel}
\usepackage{subcaption,graphicx}
\usepackage{mathtools,amsmath,amssymb,amsfonts,amsthm}
\usepackage{mathalfa}
\usepackage{tikz}
\usetikzlibrary{calc}
\usepackage{color}
\usepackage{url}
\usepackage{microtype}

\usepackage{enumitem}

\usepackage[english]{babel}

\usepackage{tabularx}

\usepackage{xcolor}

\usepackage{verbatim}
\usepackage{xspace}

\usepackage{amsfonts}
\usepackage{mathtools}
\usepackage{amsmath}

\usepackage{amssymb}

\usepackage{setspace}

\usepackage{ifthen}
\usepackage{tikz}
\usetikzlibrary{automata,positioning, arrows}

\theoremstyle{plain}
\newtheorem{thm}{Theorem}
\newtheorem{lemma}{Lemma}
\newtheorem{prop}{Proposition}
\newtheorem{cor}{Corollary}

\theoremstyle{definition}

\newtheorem{defn}{Definition}

\theoremstyle{remark}
\newtheorem{myexample}{Example}
\newtheorem{rem}{Remark}
\newtheorem{assumption}{Assumption}

\usepackage{fullpage}

\newenvironment{example}[1][0]
{ 
  \ifthenelse{\equal{#1}{0}}{
  \myexample
}
{ 
  \renewcommand\themyexample{#1}\myexample
  \addtocounter{myexample}{-1}
}
}
{\endmyexample}

\newcommand{\R}{\mathbb{R}}{}
\newcommand{\N}{\mathbb{N}}

\newcommand{\cA}{\mathcal{A}}

\newcommand{\cC}{\mathcal{C}}

\newcommand{\cI}{\mathcal{I}}
\newcommand{\cK}{\mathcal{K}}

\newcommand{\cKL}{\mathcal{KL}}

\newcommand{\cF}{\mathcal{F}}

\newcommand{\cS}{\mathcal{S}}
\newcommand{\cQ}{\mathcal{Q}}

\newcommand{\cLL}{\text{Lip}_{\ell}}

\newcommand{\wt}{\widetilde}

\newcommand{\dw}{\text{dw}}

\newcommand{\restr}[2]{{
  \left.\kern-\nulldelimiterspace 
  #1 
  \vphantom{\big|} 
  \right|_{#2} 
  }}
\title{Converse Lyapunov Results for Switched Systems with Lower and Upper Bounds on Switching Intervals}

\author{Matteo Della Rossa\thanks{The Author is with ``Dipartimento di Scienze matematiche, informatiche e fisiche'', University of Udine, Udine, Italy ).
        {\small matteo.dellarossa@uniud.it}}}

\begin{document}
\maketitle
\begin{abstract}
 The topic of this manuscript is the stability analysis of continuous-time switched nonlinear systems with constraints on the admissible switching signals. Our particular focus lies in considering signals characterized by upper and lower bounds on the length of the switching intervals. We adapt and extend the existing theory of multiple Lyapunov functions, providing converse results and thus a complete characterization of uniform stability for this class of systems. We specify our results in the context of switched linear systems, providing the equivalence of exponential stability and the existence of multiple Lyapunov norms. By restricting the class of candidate Lyapunov functions to the set of quadratic functions, we are able to provide semidefinite-optimization-based numerical schemes to check the proposed conditions. We provide numerical examples to illustrate our approach and highlight its advantages over existing methods.
\end{abstract}

\section{Introduction}
Within the broader class of hybrid dynamical systems (as detailed in \cite{goebel2012hybrid}), we focus our attention on the framework of switched systems. This setting is particularly noteworthy both in theoretical and practical perspectives. From a general point of view, switched systems exhibit a continuous-time evolution guided by a finite set of subsystems and discrete-time or jump phenomena associated with the switching between these subsystems. Formally, given $M\in \N$ and a finite set of vector fields $f_1,\dots, f_M:\R^n \to \R^n$, a \emph{switched system} is defined by
\begin{equation}\label{eq:SystemIntro}
\dot x(t)=f_{\sigma(t)}(x(t)),\;\;\;t\in \R_+.
\end{equation}
Here, $\sigma:\R_+\to \{1,\dots, M\}$ is a discrete-valued signal, referred to as the \emph{switching signal}, that models the \emph{switching} among the subsystems. For a comprehensive overview of this class of hybrid systems, we refer to~\cite{Lib03,LinAnt09,ShoWir07}.

The study of stability of~\eqref{eq:SystemIntro} has been the subject of extensive research in recent decades, both in the nonlinear and linear cases.  In both settings, the behavior of~\eqref{eq:SystemIntro} is strongly affected by the properties and assumptions concerning the class of switching signals under consideration.
When examining the class of measurable switching signals as a whole, uniform stability of~\eqref{eq:SystemIntro} is equivalent to the existence of a \emph{common} Lyapunov function, i.e. a positive definite function decreasing along \emph{any} subsystem, see~\cite{Lib03,Mancilla00} and references therein.

However, when narrowing the class of feasible switching signals, the Lyapunov characterization of uniform stability becomes non-trivial in general. One common approach to restrict the class of switching signals is by imposing a bounded time-threshold on the switching events or, equivalently, setting a minimum time between switches. These signals, known as \emph{dwell-time signals} in the literature, were introduced in the seminal paper~\cite{Morse}, and further studied and generalized in~\cite{HesMor99,Lib03}. In this setting, a mature Lyapunov function theory has emerged, offering a characterization of uniform stability in terms of \emph{multiple} Lyapunov functions, see~\cite{Wirth2005,WirthCDC05,ChiGugPro21}. Sufficient Lyapunov conditions for stability over the more general class of \emph{average dwell-time} signals can be found in~\cite[Chapter 3]{Lib03}, while numerical methods for designing such functions are discussed in~\cite{BlaCol10,GerCol06,AllSha11,CheCol12,HafTan23}.

When dealing with possibly unstable subsystems, it can be beneficial to extend the concept of dwell-time signals  by also imposing an \emph{upper bound} on the distance between switching events. In essence, given any $\tau_2\geq \tau_1\geq 0$, one can consider signals that, after any switching instant, remain constant at least $\tau_1$ time units, but no more than $\tau_2$ time units. Indeed, by bounding from above the time of permanence in any subsystem, one can ensure stability of~\eqref{eq:SystemIntro}, even if all the subsystems are unstable, see~\cite{YangJia14} for an overview. This class of signals and possible generalizations have recently been introduced and studied for example in~\cite{BriSeu13,Xiangxiao14,KunChat15,Bri16,ZhaoShi17,YangWang20,YinJay23} and references therein.  \textcolor{black}{In~\cite{Briat17} the case of positive impulsive systems under the same class of constraints on the switching/jump schedule is considered.}
\textcolor{black}{Recent results presented in~\cite{ProKam23}, provided converse Lyapunov results for \emph{linear systems} under such class of signals.
}

This framework can find applications in various physical scenarios. A field of possible application is in the context of~\emph{event-triggered control} (see \cite{HeemelsJohn12} for an overview); in this setting, the update (or switching) in the control policy is \emph{triggered} by the satisfaction (or not) of a pre-designed condition (the ``\emph{event}''). If it is known that this event occurs, after any update, not before $\tau_1$  and not after $\tau_2$ time units, then the tools developed for the previously introduced class of switching signals can be used in the analysis and control design in this scenario.

In this manuscript, we study stability of switched systems as in~\eqref{eq:SystemIntro}, under the class of signals with upper and lower \textcolor{black}{ bounds on the length of switching intervals} introduced in the previous paragraph. By adapting a proof technique provided in~\cite{WirthCDC05,Wirth2005} in the case of dwell-time signals, we are able to provide a comprehensive Lyapunov functions characterization of uniform stability. Specifically, considering a partition of the considered family of signals, we prove two (independent and novel) \emph{converse multiple Lyapunov functions} results, in a general non-linear subsystems setting. These constructions can also be interpreted, in a sense we will clarify, as  graph theory-based results. Therefore we draw comparison between our framework and the graphs-oriented findings in~\cite{ChiGugPro21} (in the context of dwell-time switched linear systems) and the results in~\cite{ahmadi,PhiEssDul16,DebDelRos22} for \emph{discrete-time} switched systems. We specify our construction in the case of switched \emph{linear} systems, providing a Lyapunov characterization of uniform exponential stability in terms of multiple Lyapunov \emph{norms}, as previously done in~\cite{Wirth2005,ChiGugPro21} for the case without upper bounds. \textcolor{black}{In this linear case, our results are compared with the recent multiple-norms characterization of exponential stability provided in~\cite{ProKam23}. Here, assuming that a  switched linear system is exponentially stable, the existence of a multi-norm decreasing at the switching instants is proven. Despite the underlying proof techniques are different, this result has strong relations with the  Lyapunov characterization proposed in this manuscript, as we discuss in what follows.
}

In order to provide conditions that are easily evaluated computationally, we further particularize our results by restricting the set of candidate Lyapunov functions to the set of quadratic functions, leading to semidefinite optimization programs (depending on some additional parameters). We show the benefit of our approach with the help of numerical examples, comparing our results with existing literature, notably~\cite{Xiangxiao14}.

The structure of the manuscript is the following: In Section~\ref{sec:Prelim} we provide the formal introduction of the studied framework, while in~Section~\ref{sec:MainResults} we illustrate our main results, in a general nonlinear setting. In Section~\ref{sec:LinearCase} we specify our approach in the case of switched linear systems and we present some numerical examples, before closing the discussion with some final remarks in Section~\ref{Sec:conclu}.
\\
\emph{Notation:}
The set $\R_+:=\{s\in \R\;\vert\;s\geq0\}$ denotes the set of non-negative real numbers.
Given $m,n\in \N$, the class  $\cLL(\R^n,\R^m)$ denotes the set of locally Lipschitz continuous functions from $\R^n$ to $\R^m$; while $\cC^1(\R^n,\R^m)$ denotes the set of continuously differentiable functions. \\
\emph{Comparison Functions Classes}:   A function $\alpha:\R_+\to \R$ is of \emph{class $\cK$} ($\alpha \in \cK$) if it is continuous, $\alpha(0)=0$, and strictly increasing; it is of \emph{class $\cK_\infty$} if, in addition, it is unbounded. A continuous function $\beta:\R_+\times \R_+\to \R_+$ is of \emph{class $\mathcal{KL}$} if $\beta(\cdot,s)$ is of class $\cK$ for all $s$, and $\beta(r,\cdot)$ is decreasing and $\beta(r,s)\to 0$ as $s\to\infty$, for all $r$. \

\section{Preliminaries}\label{sec:Prelim}
In this section we recall the main notions and definitions used in the rest of the manuscript.

Given $M\in \N$, we define $\cI:=\{1,\dots,M\}$, the \emph{index set}.
In defining the studied class of systems, we consider a set of vector fields $\cF:=\{f_1,\dots, f_M\}\subset \cLL(\R^n,\R^n)$; and we will assume the following.
\begin{assumption}\label{assumt:Regularity}
For any $i\in \cI$, we suppose that $f_i(0)=0$ and $f_i\in \cLL(\R^n,\R^n)$ is such that the corresponding \emph{subsystem} 
 \begin{equation}\label{eq:subsystem}
 \dot x=f_i(x)
 \end{equation}
exhibits existence, uniqueness and completeness (backward and forward) of solutions, denoted by $\Phi_i:\R\times \R^n\to \R^n$,~i.e.
\[
\begin{aligned}
\Phi_i(t,x):=&\text{solution to~\eqref{eq:subsystem}, starting at $x\in \R^n$},\\ &\text{ \textcolor{black}{evaluated} at time $t\in \R$.}
\end{aligned}
\]
\end{assumption}

Given a $V\in \cLL(\R^n,\R)$ and a system~\eqref{eq:subsystem} we denote by $D^+_{f_i}V$ the \emph{Dini-derivative of $V$ with respect to $f_i$}, defined by
\[
D^+_{f_i}V(x):=\limsup_{h\to 0^+}\frac{V(\Phi_i(h,x))-V(x)}{h}.
\]
We recall that if $V\in \cC^1(\R^n,\R)$ then $D^+_{f_i}V(x)=\nabla V(x)^\top f_i(x)$, for every $x\in \R^n$.

 Given a set of vector fields $\cF=\{f_1,\dots, f_M\}\subset \cLL(\R^n,\R^n)$ satisfying Assumption~\ref{assumt:Regularity} we consider the \emph{switched system} defined by
\begin{equation}\label{eq:SwitchedSystem}
\dot x(t)=f_{\sigma(t)}(x(t)),\;\;\;x(0)=x_0\in \R^n,\;\;t\in \R_+,
\end{equation}
where $\sigma:\R_+\to \cI$ is an external \emph{switching signal}. 
When all the subsystems are linear, we consider \emph{switched linear systems} defined by
\begin{equation}\label{eq:LinearSwitchedSystem}
\dot x(t)=A_{\sigma(t)}x(t),\;\;\;x(0)=x_0\in \R^n,\;\;t\in \R_+,
{}\end{equation}
where $\cA=\{A_1,\dots, A_M\}\subset \R^{n\times n}$ is a set of matrices and $\sigma:\R_+\to \cI$ is again an external switching signal. 

The switching signals $\sigma$ are selected, in general, among the set $\cS$ defined by
\begin{equation}\label{eq:arbitrarySwitching}
\cS:=\left\{\sigma:\R_+\to \cI\;\Big\vert\;\;\sigma\text{ piecewise constant and right continuous} \right\}.
\end{equation}
Given a signal $\sigma\in \cS$, we denote the sequence of switching instants, that is, the points at which $\sigma$ is discontinuous,  by $\{t^\sigma_k\}$. The set $\{t^\sigma_k \}$ may be infinite or finite, possibly reduced to the initial instant, defined by $t^\sigma_0:=0$; if it is infinite, then it is unbounded.
Given a $\sigma \in \cS$, a $x\in \R^n$ and a $t\in \R$ we denote by $\Phi_\sigma(t,x)$ the solution of~\eqref{eq:SwitchedSystem} starting at $x$ and evaluated at $t$ with respect to the switching signal {\color{black}$\sigma\in \cS$}.

We now introduce the considered concepts of stability.
\begin{defn}\label{defn:ExponentialStability}
Given a class of signal $\wt \cS\subseteq \cS$, system~\eqref{eq:SwitchedSystem} is said to be \emph{globally uniformly asymptotically  stable on $\wt \cS$} (GUAS)  if there exists a $\beta \in \cKL$ such that
\begin{equation}\label{eq:GUAS}
|\Phi_\sigma(t,x)|\leq \beta(|x|,t),\;\; \forall\,\sigma\in \wt \cS,\;\forall x\in \R^n,\;\forall \,t\in \R_+.
\end{equation}
In particular, given $\wt \cS\subseteq \cS$, system~\eqref{eq:SwitchedSystem} is said to be \emph{uniformly exponentially stable with decay $\rho>0$ on $\wt \cS$} ($\rho$-UES)  if there exists a $M\geq 0$ such that
\begin{equation}\label{eq:ExponentialStability}
|\Phi_\sigma(t,x)|\leq Me^{-\rho t}|x|,\;\;\forall\;\sigma\in \wt \cS,\;\forall x\in \R^n,\;\forall \;t\in \R_+.
\end{equation}
The supremum over the $\rho>0$ for which~\eqref{eq:ExponentialStability} is satisfied for some $M\geq 0$ is called the $\wt \cS$-\emph{exponential decay rate}, and it is denoted by $\rho_{\wt \cS}(\cF)$. 
\end{defn}

We recall that for switched \emph{linear} systems as in~\eqref{eq:LinearSwitchedSystem}, given any set $\wt \cS\subset \cS$ \textcolor{black}{closed under time right-shifting}\footnote{ \color{black}A set $\wt\cS\subseteq \cS$ is said to be \emph{closed under time right-shifting} if \\$\sigma\in \wt \cS$ $\Leftrightarrow$ $\sigma(\cdot+t)\in \wt \cS$,  $\forall t\in \R_+$.}, (GUAS) on $\wt \cS$ imply ($\rho$-UES) (for a certain $\rho>0$) on $\wt \cS$, by linearity, see for example~\cite{Lib03,AngeliNote}~\cite[Chapter 5]{BacRosier}.

\section{Main Results}\label{sec:MainResults}
In this section we focus our analysis on the class of signals with a minimal and maximal time threshold between switching instants. We provide our main converse Lyapunov results for this class of signals, providing characterization of GUAS in terms of multiple Lyapunov functions.
\subsection{Signals with Upper-Lower Switching Bounds}
We provide here the main definitions and first results on the studied class of switched signals.
\begin{defn}
Given any $\tau_2\geq \tau_1\geq 0$, consider $\cS_{dw}(\tau_1,\tau_2)$ the class of signals with \emph{lower and upper dwell-time bounds $\tau_1,\tau_2$} defined by
\begin{equation}\label{eq:StrictDwellTime}
\cS_{\dw}(\tau_1,\tau_2)\hspace{-0.05cm}:=\hspace{-0.05cm}\left\{\sigma\in \cS\vert\,\tau_1\leq t^\sigma_{k} -t^\sigma_{k-1}\leq \tau_2,\,\forall t^\sigma_k>0  \right\}.
\end{equation}
Slightly weakening the condition in~\eqref{eq:StrictDwellTime}, we also consider the class $\cS_{\dw}^\star(\tau_1,\tau_2)$, defined by 
\begin{equation}\label{eq:EnlargedDwellTime}
\cS_{\dw}^\star(\tau_1,\tau_2):=\left\{\sigma\in \cS\;\Big\vert\;\tau_1\leq t^\sigma_{k+1} -t^\sigma_{k}\leq \tau_2,\;\forall \;t^\sigma_k>0\text{ and }\;t_1^\sigma\leq \tau_2   \right\}.
\end{equation}
\end{defn}
The difference between~\eqref{eq:StrictDwellTime} and~\eqref{eq:EnlargedDwellTime} is that in~\eqref{eq:StrictDwellTime} it is required that the first switching interval i.e. $[t_0^\sigma, t_1^\sigma]=[0,t_1^\sigma]$ has length between $\tau_1$ and $\tau_2$, as any other switching interval. In~\eqref{eq:EnlargedDwellTime}, instead, we only require that $t_1^\sigma\leq \tau_2$ without imposing any lower bound on the length of the first switching interval.
One advantage of this relaxation is that, for any $\tau_1,\tau_2>0$, we have
that   $\cS_{\dw}^\star(\tau_1,\tau_2)$ is closed under  time right-shifting.\\
\textcolor{black}{
Since the mentioned difference between these classes only affects the length of the first interval, the asymptotic behavior of~\eqref{eq:SwitchedSystem} for signals in $\cS_{\dw}(\tau_1,\tau_2)$ or $\cS^\star_{\dw}(\tau_1,\tau_2)$ is substantially the same, as proven in what follows.}
\begin{lemma}\label{Lemma:Equivalence}
Consider a set of vector fields $\cF=\{f_1,\dots, f_M\}\subset \cLL(\R^n,\R^n)$ satisfying Assumption~\ref{assumt:Regularity} and $\tau_2\geq \tau_1\geq 0$. System~\eqref{eq:SwitchedSystem} is GUAS on $\cS_{\dw}(\tau_1,\tau_2)$ if and only if it is GUAS on $\cS_{\dw}^\star(\tau_1,\tau_2)$.
\\
Given any $\rho>0$, system~\eqref{eq:SwitchedSystem} is $\rho$-UES on $\cS_{\dw}(\tau_1,\tau_2)$ if and only if it is $\rho$-UES on $\cS_{\dw}^\star(\tau_1,\tau_2)$.
\end{lemma}
\begin{proof}
Since it is clear that $\cS_{\dw}(\tau_1,\tau_2)\subset \cS_{\dw}^\star(\tau_1,\tau_2)$, only one direction is to be proven.
Suppose that~\eqref{eq:SwitchedSystem} is GUAS on $\cS_{\dw}(\tau_1,\tau_2)$. Given any $\sigma\in \cS^\star_{\dw}(\tau_1,\tau_2)$, any $t\leq t^\sigma_1$ and any $x\in \R^n$ we have $|\Phi_\sigma(t,x)|\leq \beta(|x|,t)$,
since $\sigma$ is constant on $[0,t^\sigma_1)$ and thus coincide with a signal $\gamma\in \cS_{\dw}(\tau_1,\tau_2)$ in this interval.
Then, for any $t\geq t_1^\sigma$, since $\sigma(\cdot+t^\sigma_1)\in \cS_{\dw}(\tau_1,\tau_2)$ we have
\[
|\Phi_\sigma(t,x)|=|\Phi_{\sigma(\,\cdot\,+t^\sigma_1)}(t-t^\sigma_1,\Phi_\sigma(t^\sigma_1,x))|\leq \beta(|\Phi_\sigma(t^\sigma_1,x)|, t-t_1^\sigma)\leq \beta(\beta(|x|,t_1^\sigma), t-t^\sigma_1).
\]
It suffices thus to consider a $\widetilde \beta\in \cKL$ such that
$
\beta(\beta(r,s), t-s)\leq \wt \beta(r,t)$, $\forall r\in \R_+,\forall s\leq \tau_2,\;\forall t\geq s$,
and such a $\widetilde \beta$ exists, see~\cite{Kellett2014}.

The UES case is similar: suppose that~\eqref{eq:SwitchedSystem} is $ \rho$-UES on $\cS_{\dw}(\tau_1,\tau_2)$, for some $\rho>0$.
Given any $\sigma\in \cS_{\dw}^\star(\tau_1,\tau_2)$ for any $t< t_1^\sigma$  and any $x\in \R^n$ we have $|\Phi_\sigma(t,x)|\leq Me^{-\rho t}|x|$, since $\sigma$ is constant on $[0,t^\sigma_1)$. Then, for any $t\geq t_1^\sigma$, since $\sigma(\cdot+t^\sigma_1)\in \cS_{\dw}(\tau_1,\tau_2)$ computing we have
\[
\begin{aligned}
|\Phi_\sigma(t,x)|\leq Me^{-\rho(t-t_1^\sigma)}|\Phi_\sigma(t_1^\sigma, x)|\leq Me^{-\rho(t-t_1^\sigma)}Me^{-\rho t_1^\sigma}|x|=M^2 e^{-\rho t}|x|,
\end{aligned}
\]
concluding the proof.
\end{proof}
\textcolor{black}{In Lemma~\ref{Lemma:Equivalence} we proved that (asymptotic/exponential) stability with respect to  $\cS_{\dw}(\tau_1,\tau_2)$ or $\cS_{\dw}^\star(\tau_1,\tau_2)$ are equivalent. We introduced both these slightly different classes of signals since they will provide the main tool for different converse Lyapunov results, provided in the next subsection.}
\subsection{Converse Lyapunov Results}
We prove here  our first Lyapunov characterization result for GUAS of system~\eqref{eq:SwitchedSystem}. The idea behind the proof is to consider a partition of the set of signals $\cS_{\dw}(\tau_1,\tau_2)$.
\begin{thm}\label{Thm:ConverseResult1}
Consider a set of vector fields $\cF=\{f_1,\dots, f_M\}\subset \cLL(\R^n,\R^n)$ satisfying Assumption~\ref{assumt:Regularity}. Given $\tau_2\geq \tau_1\geq 0$, system~\eqref{eq:SwitchedSystem} is GUAS on $\cS_{\dw}(\tau_1,\tau_2)$ (and thus on~$\cS^\star_{\dw}(\tau_1,\tau_2)$) if and only if there exist $\alpha_1,\alpha_2\in \cK_\infty$  and continuous functions $V_i^-, V_i^+:\R^n\to \R$, for $i\in \cI$, such that 
\begin{subequations}\label{eq:ConditionsLowerUpper}
\begin{align}
\alpha_1(|x|)\leq V_i^{*}(x)\leq \alpha_2(|x|),\;\;\;&\forall i\in \cI,\,\forall *\in \{-,+\},\forall \,x\in \R^n,
\label{eq:UpperLowerSandwich}\\ \noalign{\vskip4pt}
V_i^+(\Phi_i(t,x))\leq e^{- t}V_i^-(x),\;\;\;&\forall i\in \cI,\forall t\in [0,\tau_2-\tau_1],\;\forall\,x\in \R^n,\label{eq:UpperLowerInequaility1}\\\noalign{\vskip4pt}
\hspace{-0.8cm}V_j^-(\Phi_i(\tau_1,x))\leq e^{-\tau_1}V_i^+(x),\;\;\;&\forall\,i\neq j\in \cI,\;\forall\;x\in \R^n.\label{eq:UpperLowerInequaility2}
\end{align}
\end{subequations}
\end{thm}
\begin{proof}
($\Leftarrow$:) Suppose there exist functions $V_i^-,V_i^+:\R^n\to \R$, $i\in \cI$, satisfying~\eqref{eq:ConditionsLowerUpper}.
Consider any $\sigma\in \cS_{\dw}(\tau_1,\tau_2)$, we will construct a function $U_\sigma:\R_+\times \R^n\to \R$ decreasing along any solution. Figure~\ref{figure:Figure1} provides a graphical illustration of the idea behind the subsequent construction. We will proceed by defining $U_\sigma$ by steps, in any interval of the form $[t_k^\sigma, t^\sigma_{k+1}]$, $k\in \N$. Suppose $\sigma(t^\sigma_0)=\sigma(0)=i$, $t^\sigma_1\in [\tau_1,\tau_2]$, and $\sigma(t^\sigma_1)=j\neq i$.
First, using the property~\eqref{eq:UpperLowerInequaility2}, let us consider a continuous function $U_{ij}:[0,\tau_1]\times \R^n\to \R$, defined as in Lemma~\ref{lem:IntermadiateLemma} in Appendix, satisfying
\[
\begin{aligned}
&U_{ij}(0,x)=V_i^+(x)\;\;\wedge\; \;U_{ij}(\tau_1,x)=V_j^-(x),\;\;\;\forall \,x\in \R^n,\\
&U_{ij}(t,\Phi_i(t,x))\leq e^{-t}U_{ij}(0,x),\;\;\;\;\forall\,(t,x)\in[0,\tau_1]\times \R^n.
\end{aligned}
\]
Let us define $U_\sigma$ on $[0,t_1^\sigma]\times \R^n$ by
\[
U_\sigma(t,x):=\begin{cases}
V^-_i(x)\;\;\;\;&\text{if }t=0,\\
V^+_i(x)\;\;\;\;&\text{if }t\in (0,t_1^\sigma-\tau_1],\\
U_{ij}(t-t_1^\sigma+\tau_1,x)\;\;\;\;&\text{if }t\in (t_1^\sigma-\tau_1,t^\sigma_1];
\end{cases}
\]
note that, by Lemma~\ref{lem:IntermadiateLemma}, we have $U_\sigma(t_1^\sigma,x)=V_j^-(x)$.
Using~\eqref{eq:UpperLowerInequaility1} and the aforementioned properties of the function $U_{ij}$ we have 
\[
U_\sigma(t,\Phi_i(t,x))\leq e^{-t}U_\sigma(0,x)=e^{-t}V^-_i(x),
\]
for all $(t,x)\in [0,t_1^\sigma]\times \R^n$.
We can iterate the construction on any interval of the form $[t_k^\sigma,t^\sigma_{k+1}]$, $k\in \N$, obtaining a function $U_\sigma:\R_+\times\R^n\to \R$ such that
\[
U_\sigma(t,\Phi_\sigma(t,x))\leq e^{-t}U_\sigma(0,x),\;\;\forall \;(t,x)\in \R_+\times \R^n.
\]
Moreover, using~\eqref{eq:UpperLowerSandwich} and Lemma~\ref{lem:IntermadiateLemma} it can be seen that there exist $\wt \alpha_1, \wt \alpha_2\in \cK_\infty$ (not depending on $\sigma\in \cS_{\dw}(\tau_1,\tau_2)$) such that
$
\wt \alpha_1(|x|)\leq U_\sigma(t,x)\leq \wt \alpha_2(|x|)$, for all $(t,x)\in \R_+\times \R^n$.
Summarizing,  we have
\[
\begin{aligned}
|\Phi_\sigma(t,x)|\leq \wt \alpha^{-1}_1\left(U_\sigma(t, \Phi_\sigma(t,x)) \right)\leq \wt \alpha^{-1}_1\left(e^{-t}U_\sigma(0,x) \right)\leq \wt \alpha^{-1}_1\left(e^{-t}\wt \alpha_2(|x|) \right),
\end{aligned}
\]
for any $\sigma \in\cS_{\dw}(\tau_1,\tau_2)$, any $t\in \R_+$ and any $x\in \R^n$. This concludes the proof, since the function defined by $\wt \beta(r,s):=\wt \alpha^{-1}_1(e^{-s}\wt \alpha_2(r))$ is of class $\cKL$ and it does not depend on $\sigma \in\cS_{\dw}(\tau_1,\tau_2)$.
\\
($\Rightarrow$:) Suppose that system~\eqref{eq:SwitchedSystem} is GUAS on $\cS_{\dw}(\tau_1,\tau_2)$ and thus there exists a $\beta\in\cKL$ such that~\eqref{eq:GUAS} holds.
First of all,  using \cite[Proposition 7]{Son98} (see also~\cite[Lemma 3]{TeelPraly2000}) there  exist functions $\rho_1,\rho_2\in \cK_\infty$ such that
$
\beta(s,t)\leq \rho_1(e^{-t}\rho_2(s))$, $\forall s\in \R_+,\;\;\forall\;t\in \R_+$,
and thus we have
\begin{equation}\label{eq:InequalityConverseLemma}
|\Phi_\sigma(t,x)|\leq \rho_1(e^{-t}\rho_2(|x|)),
\end{equation}
for all $\sigma\in \cS_{\dw}(\tau_1,\tau_2)$, for all $x\in \R^n$ and for all $t\in \R_+$. 
Let us consider, for every $i\in \cI$, the subfamilies of signals defined by
\begin{equation}\label{eq:SpecialClassOfSignals}
\begin{aligned}
\cS^-(i)&:=\{\gamma\in \cS_{\dw}(\tau_1,\tau_2) \;\vert\; \gamma(0)=i \},\\
\cS^+(i)&:=\{\gamma\in \cS_{\dw}(\tau_1,\tau_2)\;\vert\; \gamma(0)=i\;\wedge \; t^\gamma_1=\tau_1\}.
\end{aligned}
\end{equation}
We then define $V_i^-,V_i^+:\R^n\to \R$, $i\in \cI$ by
\[
\begin{aligned}
V_i^-(x):=\sup_{t\geq 0, \,\sigma \in \cS^-(i)}\{e^{t}\rho_1^{-1}(|\Phi_\sigma(t,x)|)\; \},\\
V_i^+(x):=\sup_{t\geq 0, \,\sigma \in \cS^+(i)}\{e^{t}\rho_1^{-1}(|\Phi_\sigma(t,x)|)\; \}.
\end{aligned}
\]
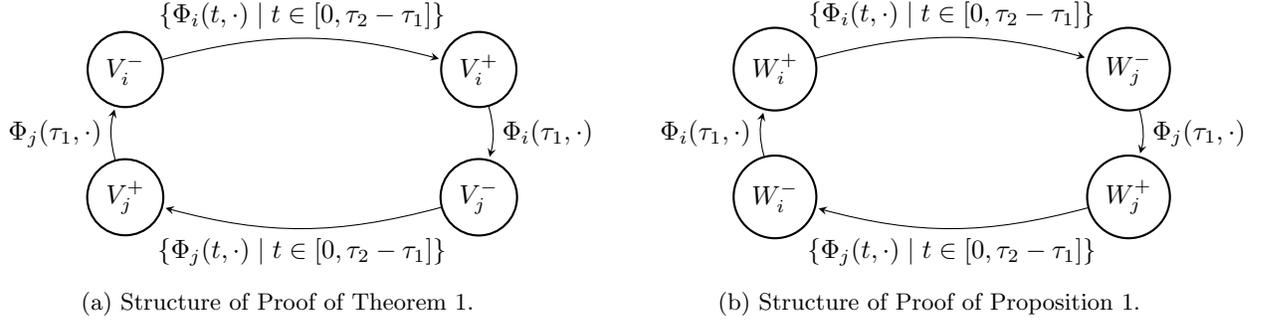
\begin{figure*}[t!]
\centering
\vspace{1cm}
\begin{subfigure}{0.45\linewidth}
  \color{black}
  \centering
  \begin{tikzpicture}%
  [>=stealth,
  shorten >=1pt,
  node distance=1cm,
  on grid,
  auto,
  state/.append style={minimum size=1cm},
  every state/.style={draw=black, fill=white, thick}
  ]
  \node[state] (left)  [yshift=-0.3cm]                {$V_i^-$};
  \node[state] (right) [right=of left, xshift=3.7cm] {$V_i^+$};
  \node[state] (belowleft) [below=of left, yshift=-0.7cm]{$V_j^+$};
   \node[state] (belowright) [below=of right,  yshift=-0.7cm]{$V_j^-$};
  \path[->]
   (left) edge[bend left=15]     node                      {$\{\Phi_i(t,\cdot)\;\vert\;t\in [0,\tau_2-\tau_1]\}$} (right)
  (right) edge[bend left=15]     node                      {$\Phi_i(\tau_1,\cdot)$} (belowright)
 (belowright) edge[bend left=15]     node                      {$\{\Phi_j(t,\cdot)\;\vert\;t\in [0,\tau_2-\tau_1]\}$} (belowleft)
(belowleft)  edge[bend left=15]     node                      {$\Phi_j(\tau_1,\cdot)$} (left)
;
\end{tikzpicture}
\caption{ Structure of Proof of Theorem~\ref{Thm:ConverseResult1}.}
\end{subfigure}
\hspace{1cm}
\begin{subfigure}{0.45\linewidth}
  \color{black}
  \centering
  \begin{tikzpicture}%
  [>=stealth,
  shorten >=1pt,
  node distance=1cm,
  on grid,
  auto,
  state/.append style={minimum size=1.1cm},
  every state/.style={draw=black, fill=white, thick}
  ]
  \node[state] (left)  [yshift=-0.3cm]                {$W_i^+$};
  \node[state] (right) [right=of left, xshift=3.7cm] {$W_j^-$};
  \node[state] (belowleft) [below=of left, yshift=-0.7cm]{$W_i^-$};
   \node[state] (belowright) [below=of right,  yshift=-0.7cm]{$W_j^+$};
  \path[->]
   (left) edge[bend left=15]     node                      {$\{\Phi_i(t,\cdot)\;\vert\;t\in [0,\tau_2-\tau_1]\}$} (right)
  (right) edge[bend left=15]     node                      {$\Phi_j(\tau_1,\cdot)$} (belowright)
 (belowright) edge[bend left=15]     node                      {$\{\Phi_j(t,\cdot)\;\vert\;t\in [0,\tau_2-\tau_1]\}$} (belowleft)
(belowleft)  edge[bend left=15]     node                      {$\Phi_i(\tau_1,\cdot)$} (left)
;
\end{tikzpicture}
\caption{ Structure of Proof of Proposition~\ref{Prop:ConverseResult2}.}
\end{subfigure}
\caption{The graph representations of the partitions and Lyapunov constructions used in Theorem~\ref{Thm:ConverseResult1} and Proposition~\ref{Prop:ConverseResult2}. Any arrow stands for a required inequality: for instance, an edge from the node $U_a$ to the node $U_b$ labeled by a (set of) operator(s) $\Phi_s(\tau,\cdot)$ represents the inequality $U_b(\Phi_s(\tau,x))\leq e^\tau U_a(x)$, $\forall \,x\in \R^n$. }
\label{figure:Figure1}
\end{figure*}
Inequality~\eqref{eq:UpperLowerSandwich} is straightforward, since $\cS^*(i)\subset \cS_{\dw}(\tau_1,\tau_2)$, for every $\forall i\in \cI,\;\forall *\in \{-,+\}$ and recalling~\eqref{eq:InequalityConverseLemma}, we have
$
\rho_1^{-1}(|x|)\leq V_i^\ast(x)\leq \rho_2(|x|)$, for all $x\in \R^n$.
Consider now any $\sigma\in \cS^+(i)$  and any $\tau\in [0,\tau_2-\tau_1]$; we define $\gamma_{\sigma,\tau}\in \cS$ by
\[
\gamma_{\sigma,\tau}(s)=\begin{cases}
i \;\;\;&\text{if } s< \tau,\\
\sigma(s-\tau) &\text{if } s\geq \tau;
\end{cases}
\]
 it is clear that $\gamma_{\sigma,\tau}\in \cS^-(i)$. Thus, considering $i\in \cI$, $\tau\in [0, \tau_2-\tau_1]$ and $x\in \R^n$, computing we have
 \[
 \begin{aligned}
V^+_i(\Phi_i(\tau,x))&=\sup_{t\geq 0, \,\sigma \in \cS^+(i)}\{e^{t}\rho_1^{-1}(|\Phi_\sigma(t,\Phi_i(\tau,x))|)\; \}=\sup_{t\geq 0, \,\sigma \in \cS^+(i)}\{e^{t}\rho_1^{-1}(|\Phi_{\gamma_{\sigma,\tau}}(t+\tau,x)|)\; \}\\&\leq \sup_{s\geq 0, \,\gamma \in \cS^-(i)}\{e^{s-\tau}\rho_1^{-1}(|\Phi_{\gamma}(s,x)|)\; \}=e^{-\tau}V^-_i(x),
\end{aligned}
 \]
 proving~\eqref{eq:UpperLowerInequaility1}. Now, consider any $j\neq i\in \cI$ and any $x\in \R^n$, with a similar reasoning, we have
 \[
 \begin{aligned}
V^-_j(\Phi_i(\tau_1,x))=\sup_{t\geq 0,\,\sigma\in \cS^-(j)}\{e^{t}\rho_1^{-1}(|\Phi_\sigma(t,\Phi_i(\tau_1,x))|)\;| \}\leq \sup_{s\geq 0,\,\gamma\in \cS^+(i)}\{e^{s-\tau_1 }\rho_1^{-1}(|\Phi_\gamma(s,x)|)\;| \}
  \leq e^{-\tau_1}V_i^+(x),
 \end{aligned}
 \]
 concluding.
 The continuity of $V_i^{+}, V_i^{-}$ for $i\in \cI$ is not proven here, we refer to~\cite[Proposition~5]{TeelPraly2000} for the technical argument. 
\end{proof}

In what follows, we discuss the relations of the conditions of Theorem~\ref{Thm:ConverseResult1} and existing results, in some particular cases.
\begin{rem}[Limiting cases] In the case $\tau_1=\tau_2=\wt \tau>0$ the class of signals $\cS_{\dw}(\tau_1,\tau_2)$ corresponds to the class of signals with a discontinuity at each point of the form $k\wt \tau$, for any $k\in \N$, called here the \emph{class of $ \wt \tau$-fixed-time switching signals}, formally defined by
\[
\cS_{\text{fix}}(\wt\tau)=\cS_{\dw}(\wt \tau, \wt \tau)=\{\sigma \in \cS\;\vert\; t^\sigma_{k+1}-t^\sigma_{k}=\wt \tau, \;\forall k\in \N \}.
\]
The asymptotic behavior of system~\eqref{eq:SwitchedSystem} on this class is completely determined by the behavior of the \emph{discrete-time} switched system defined by
\[
x(k+1)=\Phi_{\gamma(k)}(\wt \tau, x(k))
\]
with $\gamma:\N\to \cI$ a \emph{discrete-time} signal such that $\gamma(k)\neq \gamma(k+1)$, $\forall\;k\in \N$.
  For that reason, the statement of Theorem~\ref{Thm:ConverseResult1} can be seen as a particular case of the conditions presented in~\cite{ahmadi,PhiEssDul16} in the context of constrained discrete-time switched systems. More precisely, conditions~\eqref{eq:UpperLowerInequaility1}~\eqref{eq:UpperLowerInequaility2} in this case read
  \[
\begin{aligned}
V_i^+(x)\leq V_i^-(x),\;\;\;&\forall\, i\in \cI,\;\forall \;x\in \R^n,\\
V_j^-(\Phi_i(\wt \tau,x))\leq e^{-\wt \tau}V_i^+(x),\; \;\;&\forall\, i\neq j\in \cI,\;\forall \;x\in \R^n,
\end{aligned}
\]
and further identifying $V_i^-\equiv V_i^+$ for any $i\in \cI$,  the conditions turn out to require the existence of continuous functions $\wt V_1,\dots, \wt V_M:\R^n\to \R$ such that
\[
\wt V_j(\Phi_i(\wt \tau,x))\leq e^{-\wt \tau}\wt V_i(x),\; \;\;\forall\, i\neq j\in \cI,\;\forall \;x\in \R^n,
\]
and satisfying inequalities of the form~\eqref{eq:UpperLowerSandwich}. This corresponds, in a non-linear setting, to the conditions illustrated in~\cite{PhiEssDul16} for the class of discrete-time signals with no consecutive occurrences of the same symbol.

\textcolor{black}{
In the limiting case $\tau_2\to \infty$ (for a fixed $\tau_1\geq 0$), we have that $\cS_{\dw}(\tau_1,\tau_2)$ somehow approaches the class of \emph{dwell-time signals} 
\begin{equation}\label{Eq:PuerlyDwellTime}
\overline \cS_{\dw}(\tau_1):=\left\{\sigma\in \cS\;\vert\;\tau_1\leq t^\sigma_{k} -t^\sigma_{k-1},\;\forall \;t^\sigma_k>0  \right\},
\end{equation} introduced in the seminal reference~\cite{Morse} and intensively studied since then, see for example~\cite{HesMor99,Lib03,Wirth2005,WirthCDC05,GerCol06,ChiGugPro21} and references therein.} {\color{black}More precisely, for every $T\in \R_+$, there exists a $\tau_2\geq \tau_1$ such that  
\[
 \overline \cS_{\dw}(\tau_1)_{\vert_T}=\cS_{\dw}(\tau_1,\tau_2)_{\vert_{T}},
\]
where, given $\wt \cS\subseteq \cS$, we denote by $\wt \cS_{\vert_T}$ the set of $T$-restriction of signals in $\wt \cS$, i.e., $\wt \cS_{\vert_T}=\{\gamma:[0,T)\to \mathcal{I}\;\vert\; \exists \,\sigma\in \wt \cS\text{ s.t. } \gamma=\sigma_{\vert_{[0,T)}}\}$
}
In this case, for the sake of simplicity setting $V_i^+\equiv V_i^-$, for any $i\in \cI$, conditions~\eqref{eq:UpperLowerInequaility1}~\eqref{eq:UpperLowerInequaility2} in Theorem~\ref{Thm:ConverseResult1}  turn out to be 
\[
\begin{aligned}
\wt W_i(\Phi_i(t,x))&\leq e^{-t} \wt W_i(x)\;\;\forall \,i\in \cI,\;\forall \;x\in \R^n,\forall \;t\in \R_+,\\
\wt W_j(\Phi_i(\tau_1,x))&\leq e^{-\tau_1} \wt W_i(x)\;\;\forall i\neq j\in \cI,\;\;\forall \;x\in \R^n.
\end{aligned}
\] 
for some continuous functions $\wt W_1,\dots, \wt W_M:\R^n\to \R$ satisfying bounds as in~\eqref{eq:UpperLowerSandwich}. \textcolor{black}{These conditions are the specification, in the non-linear case, of the necessary and sufficient criteria introduced in~\cite{Wirth2005,WirthCDC05} for exponential stability on $\overline \cS_{\dw}(\tau_1)$, see also~\cite{GerCol06},~\cite{COLGer2008} for  LMIs results with the same structure.} This correspondence was somehow expected since the proof technique of~Theorem~\ref{Thm:ConverseResult1} was inspired by the arguments used in these references.
Summarizing, in the two ``limiting cases'' of fixed-time switching ($\tau_1=\tau_2$) and dwell-time signals ($\tau_2\to +\infty$), Theorem~\ref{Thm:ConverseResult1} recovers classic results already presented in the literature.
\end{rem}
We now provide an alternative Lyapunov result, which arises from a partition of signals in~$\cS^\star_{\dw}(\tau_1,\tau_2)$ (instead of $\cS_{\dw}(\tau_1,\tau_2)$). The proposed conditions are again necessary and sufficient, and different for those of Theorem~\ref{eq:SwitchedSystem}.
\begin{prop}\label{Prop:ConverseResult2}
Consider a set of vector fields $\cF=\{f_1,\dots, f_M\}\subset \cLL(\R^n,\R^n)$ satisfying Assumption~\ref{assumt:Regularity}. Given $\tau_2\geq \tau_1\geq 0$, system~\eqref{eq:SwitchedSystem} is GUAS on $\cS^\star_{\dw}(\tau_1,\tau_2)$ (and, thus on $\cS_{\dw}(\tau_1,\tau_2)$) if and only if there exist $\alpha_1,\alpha_2\in \cK_\infty$  and continuous functions $W_i^-, W_i^+:\R^n\to \R$,  such that 
\begin{subequations}\label{eq:ConditionsLowerUpper*}
\begin{align}
&\alpha_1(|x|)\leq W_i^{*}(x)\leq \alpha_2(|x|),\;\;\forall i\in \cI,\,\forall *\in \{-,+\},\;\forall \,x\in \R^n,\label{eq:UpperLowerSandwich2}
\\\noalign{\vskip4pt}
&W_i^+(\Phi_i(\tau_1,x))\leq e^{- \tau_1}W_i^-(x),\;\;\forall\,i\in \cI,\;\forall\;x\in \R^n,\label{eq:UpperLowerInequaility221}\\\noalign{\vskip4pt}
&W_j^-(\Phi_i(t,x))\leq e^{-t}W_i^+(x),\;\;\forall i\neq j\in \cI,\forall t\in [0,\tau_2-\tau_1],\;\forall\;x\in \R^n.\label{eq:UpperLowerInequaility222}
\end{align}
\end{subequations}
\end{prop}
\begin{proof}
The structure of the proof is depicted in Figure~\ref{figure:Figure1}.
For the sufficiency of existence of functions $W_i^*$ as in the statement, one can proceed as in proof of Theorem~\ref{Thm:ConverseResult1} constructing, for any $\sigma\in \cS_{\dw}^\star(\tau_1,\tau_2)$, a function $U_\sigma:\R_+\times \R^n\to \R$ decreasing along the solutions and whose decreasing decay and $\cK_\infty$ bounds are independent to $\sigma\in \cS_{\dw}^\star(\tau_1,\tau_2)$, and thus implying the GUAS property of system~\eqref{eq:SwitchedSystem} on $\cS_{\dw}^\star(\tau_1,\tau_2)$.\\
For the necessity, as in proof of Theorem~\ref{Thm:ConverseResult1}, we suppose that system~\eqref{eq:SwitchedSystem} is GUAS on $\cS_{\dw}^\star(\tau_1,\tau_2)$ and thus there  exist functions $\rho_1,\rho_2\in \cK_\infty$ such that
\begin{equation}\label{eq:InequalityConverseLemma2}
|\Phi_\sigma(t,x)|\leq \rho_1(e^{- t}\rho_2(|x|)),
\end{equation}
for all $\sigma\in \cS_{\dw}^\star(\tau_1,\tau_2)$, for all $x\in \R^n$ and for all $t\in \R_+$. 
Let us define
\[
\begin{aligned}
\cQ^-(i)&:=\{\sigma\in \cS^\star_{\dw}(\tau_1,\tau_2)\;\vert\;\sigma(0)=i\;\;\wedge t^\sigma_1\geq \tau_1\},\\
\cQ^+(i)&:=\{\sigma\in \cS^\star_{\dw}(\tau_1,\tau_2)\;\vert\;\sigma(0)=i\;\;\wedge\;\;t_1^\sigma\leq \tau_2-\tau_1\}.
\end{aligned}
\]
We define the function $W_i^-,W_i^+:\R^n\to \R$ by
\[
\begin{aligned}
W_i^-(x):=\sup_{t\geq 0, \,\sigma \in \cQ^-(i)}\{e^{t}\rho_1^{-1}(|\Phi_\sigma(t,x)|) \}.\\
W_i^+(x):=\sup_{t\geq 0, \,\sigma \in \cQ^+(i)}\{e^{ t}\rho_1^{-1}(|\Phi_\sigma(t,x)|)\; \}.
\end{aligned}
\]
Inequality~\eqref{eq:UpperLowerSandwich2} again follows by~\eqref{eq:InequalityConverseLemma2}.
Now consider any $\sigma\in \cQ^+(i)$, it is clear that the signal 
$\eta_{\sigma}\in \cS$ defined by
\[
\eta_{\sigma}(s)=\begin{cases}
i \;\;\;&\text{if } s< \tau_1,\\
\sigma(t-\tau_1) &\text{if } s\geq \tau_1;
\end{cases}
\]
is a signal in~$\cQ^-(i)$, and thus, for any $x\in \R^n$, we have
\[
\begin{aligned}
W^+_i(\Phi_i(\tau_1,x))&=\sup_{t\geq 0, \,\sigma \in \cQ^+(i)}\{e^{ t}\rho^{-1}(|\Phi_\sigma(t,\Phi_i(\tau_1,x))|)\; \}=\sup_{t\geq 0, \,\sigma \in \cQ^+(i)}\{e^{ t}\rho^{-1}(|\Phi_{\eta_{\sigma}}(t+\tau_1,x)|)\; \}\\
& \leq \sup_{s\geq 0, \gamma\in \cQ^-(i)}\{e^{(s-\tau_1)}\rho^{-1}(|\Phi_\gamma(s,x)|)\; \}=e^{-\tau_1}W^-_i(x).
\end{aligned}
\]
Computing similarly, for any $\tau\in [0,\tau_2-\tau_1]$ we have
\[
\begin{aligned}
W^-_j(\Phi_i(\tau,x))=\sup_{t\geq 0, \,\sigma \in \cQ^-(j)}\{e^{ t}\rho^{-1}(|\Phi_\sigma(t,\Phi_i(\tau,x))|)\; \} \leq \sup_{s\geq 0, \,\gamma \in \cQ^+(i)}\{e^{ (s-\tau)}\rho^{-1}(|\Phi_\gamma(s,x)|)\; \}=e^{-\tau}W^+_i(x).
\end{aligned}
\]
We conclude, referring to~\cite{TeelPraly2000} for the continuity of the functions $W_i^-,W_i^+$.
\end{proof}
\begin{rem}[Possible Generalizations and Numerical Verification of the Conditions]
In Theorem~\ref{Thm:ConverseResult1} and Proposition~\ref{Prop:ConverseResult2} we provided two different Lyapunov criteria, and we proved that they provide sufficient and necessary conditions for GUAS on $\cS_{\dw}(\tau_1,\tau_2)$ (and on $\cS_{\dw}^\star(\tau_1,\tau_2)$). Theorem~\ref{Thm:ConverseResult1} is based on a partition of the set $\cS_{\dw}(\tau_1,\tau_2)$, while the foundation of Proposition~\ref{Prop:ConverseResult2} is represented by a partition of the class $\cS^\star(\tau_1,\tau_2)$. These partitions have a graph-based interpretation, provided in Figure~\ref{figure:Figure1}, and our results can thus be interpreted/re-stated in the graph-based formalism of~\cite{ChiGugPro21}, in a non-linear and upper-and-lower bounds case. This idea was also inspired by related results in the context of discrete time switched systems, as the ones in~\cite{ahmadi,DebDelRos22,PhiEssDul16}. Different partitions/different graph representations can thus provide other possible sufficient (and, is some cases, necessary) Lyapunov conditions for GUAS on~$\cS_{\dw}(\tau_1,\tau_2)$; this line of research is not further explored here, and it remains open for future research. Instead,  we decided to present only the conditions in Theorem~\ref{Thm:ConverseResult1} and Proposition~\ref{Prop:ConverseResult2} since they arise from basic partitions of  $\cS_{\dw}(\tau_1,\tau_2)$ and $\cS^\star_{\dw}(\tau_1,\tau_2)$, and can thus be considered as standard cases, that one can easily generalize to other signal partitions/graph representations.

From a numerical point of view, the explicit computation/construction of functions satisfying~\eqref{eq:ConditionsLowerUpper} or~\eqref{eq:ConditionsLowerUpper*} is not straightforward, since these inequalities involve the flow maps associated to the subsystems $f_1,\dots, f_M$ and thus require the explicit computation of solutions, in a general case. 
In Appendix, in Lemmas~\ref{lemma:AppendixAuxiliary2} and~\ref{lemma:EasyImplication}, we provide results that can be helpful in relaxing these conditions, circumnavigating the explicit dependence on the subsystems solutions. In the following section, we  expand the discussion in the linear case.
\end{rem}

\section{Linear Case}\label{sec:LinearCase}
In this section we particularize the results presented in Section~\ref{sec:MainResults} in the linear case, i.e., when all the subsystems are linear. Moreover, we provide numerically tractable sufficient conditions  when restricting the search of the Lyapunov functions to the class of quadratics.
First, we present the converse Lyapunov result descending from Theorem~\ref{Thm:ConverseResult1}. We avoid, for the sake of concision, to explicitly present the ``translation'', in the linear case, of Proposition~\ref{Prop:ConverseResult2}, since this can be done as for Theorem~\ref{Thm:ConverseResult1}, \emph{mutatis mutandis}.
\begin{cor}\label{cor:LinearCase}
Consider $\cA=\{A_1,\dots, A_M\}\subset\R^{n\times n}$ and $\tau_2\geq \tau_1\geq 0$. We have that  $\rho=\rho_{\cS_\dw(\tau_1,\tau_2)}(\cA)$ is the $\cS_{\dw}(\tau_1,\tau_2)$-exponential decay rate of system~\eqref{eq:LinearSwitchedSystem}, as introduced in Definition~\ref{defn:ExponentialStability}, if and only if for all $\alpha<\rho$ there exist \emph{norms} $v_i^-, v_i^+:\R^n\to \R$, $i\in \cI$, such that
\begin{subequations}\label{eq:ConditionsLinearFirst}
\begin{align}
&v_i^+(e^{A_it}x)\leq e^{-\alpha t}v_i^-(x),\;\begin{aligned}&\forall\,i\in \cI,\,\forall t\in [0,\tau_2-\tau_1],\\&\forall\;x\in \R^n,
\end{aligned}
\label{eq:COnditionsLinearFirst1}\\\noalign{\vskip8pt}
&v_j^-(e^{A_i\tau_1}x)\leq e^{-\alpha \tau_1}v_i^+(x),\;\;\forall\,i\neq j\in \cI,\,\forall\,x\in \R^n.\label{eq:COnditionsLinearFirst2}
\end{align}
\end{subequations}
\end{cor}
\begin{proof}
The statement is a direct consequence of Theorem~\ref{Thm:ConverseResult1} and it is inspired by the construction in~\cite{Wirth2005,WirthCDC05}. The sufficiency, i.e. the fact that existence of norms as in~\eqref{eq:ConditionsLinearFirst} implies $\alpha$-UES of system~\eqref{eq:LinearSwitchedSystem}  can be proven with the same steps as in proof of Theorem~\ref{Thm:ConverseResult1}.
For the necessity, let us fix any $\alpha<\rho$, by definition of the exponential decay rate, there exists $M>0$ such that $|\Phi_\sigma(t,x)|\leq Me^{-\alpha t}|x|$, for all $\sigma\in \cS_{\dw}(\tau_1,\tau_2)$, all $x\in \R^n$ and all $t\in \R_+$. Then, we proceed as in proof of Theorem~\ref{Thm:ConverseResult1}, considering the subsets of signals $\cS^-(i), \cS^+(i)$ as in~\eqref{eq:SpecialClassOfSignals}.
We then define 
\[
\begin{aligned}
v_i^-(x)&:=\sup_{t\geq 0, \,\sigma \in \cS^-(i)}\{e^{\alpha t}|\Phi_\sigma(t,x)|\; \},\\
v_i^+(x)&:=\sup_{t\geq 0, \,\sigma \in \cS^+(i)}\{e^{\alpha t}|\Phi_\sigma(t,x)|\; \}.
\end{aligned}
\]
The properties~\eqref{eq:ConditionsLinearFirst} can be checked as in proof of Theorem~\ref{Thm:ConverseResult1}, it remains to verify that $v_i^-,v_i^+$ are norms, for every $i\in \cI$.
First of all, since $|\Phi_\sigma(t,x)|\leq Me^{-\alpha t}|x|$ we have that $|x|\leq v_i^{*}(x)\leq M|x|$ for any $x\in \R^n$ any $i\in \cI$ and any $*\in \{-,+\}$.  Then by linearity we have, given any $\lambda\in \R$, that
\[
\begin{aligned}
v_i^-(\lambda x)=\sup_{t\geq 0, \,\sigma \in \cS^-(i)}\{e^{\alpha t}|\Phi_\sigma(t,\lambda x)|\}=|\lambda|\sup_{t\geq 0, \,\sigma \in \cS^-(i)}\{e^{\alpha t}|\Phi_\sigma(t,x)|\}=|\lambda|v^-_i(x),
\end{aligned}
\]
and similarly for $v_i^+$. For the triangle inequality we have
\[
\begin{aligned}
v^-_i(x+y)&=\sup_{t\geq 0, \,\sigma \in \cS^-(i)}\{e^{\alpha t}|\Phi_\sigma(t, x+y)|\}=\sup_{t\geq 0, \,\sigma \in \cS^-(i)}\{e^{\alpha t}|\Phi_\sigma(t,x)+\Phi_\sigma(t,y)|\}\\
&\leq \sup_{t\geq 0, \,\sigma \in \cS^-(i)}\{e^{\alpha t}|\Phi_\sigma(t,x)|\} +\sup_{t\geq 0, \,\sigma \in \cS^-(i)}\{e^{\alpha t}|\Phi_\sigma(t,y)|\}= v^-_i(x)+v^-_i(y),
\end{aligned}
\]
and similarly for $v_i^+$, concluding the proof.
\end{proof}

{\color{black}
\begin{rem}
The Lyapunov characterization of UES for switched linear systems on $\cS_{\dw}(\tau_1,\tau_2)$ is the central topic of the recent article~\cite{ProKam23}. In this reference, starting from the notion of \emph{Lyapunov exponent} (corresponding to $-\rho_{\cS_\dw(\tau_1,\tau_2)}(\cA)$ in the notation of this submission), a family of $M$ norms, $v_1,\dots, v_M$, is build. Under the UES assumption, these norms, referred to as \emph{Lyapunov multinorms}, have the property that the function defined by $w_{\sigma,x}(t)=v_{\sigma(t)}(\Phi_\sigma(t,x))$ is decreasing \emph{at switching times}, for any $\sigma\in \cS_{\dw}(\tau_1,\tau_2)$ and any $x\in\R^n$. Despite the proof technique of~\cite{ProKam23} in obtaining these norms is distinct from the partition-based proof of this submission,  the following relation can be highlighted. Considering the functions $v^-_1,\dots, v^-_M$ in Corollary~\ref{cor:LinearCase}, one recovers a set of Lyapunov multinorms (as defined in~\cite{ProKam23}), while the norms $v^+_1,\dots, v^+_M$, can be considered as auxiliary norms, providing an estimation of the behavior of the solutions \emph{between switching instants}. Further connections between these two different Lyapunov converse constructions are still under investigation. We mention that the analysis in~\cite{ProKam23} is also specialized in the case of~\emph{irreducible} matrices, providing the existence of Barabanov/invariant multinorms. Moreover, based on a reduction to periodic switching signals, an algorithmic scheme to approximate these norms (by polyhedral norms) is also provided. The peculiarity of Corollary~\ref{cor:LMIConditions} with respect to the results in~\cite{ProKam23} are further analyzed in what follows, where we provide numerical techniques/schemes in order to verify the proposed conditions.
\end{rem}
}
Corollary~\ref{cor:LinearCase} can be used as a criterion for establishing exponential stability of a given linear switched systems, once the parameter $\tau_1,\tau_2$ and the matrices $A_1,\dots, A_M$ are given. Since the search/optimization over the set of norms is generally unfeasible from the numerical point of view, the search of functions satisfying conditions of Corollary~\ref{cor:LinearCase} can be restricted over particular subclasses of norms/functions.
As an example, one can consider SOS polynomia (as done in~\cite{CheCol12} for the dwell-time case), polyhedral functions (as defined in~\cite{BlaCol10}), etc.
In what follows, we specify the conditions of Corollary~\ref{cor:LinearCase} restricting the search to quadratic norms, i.e. functions of the form $w(x)=\sqrt{x^\top Qx}$ for a positive define matrix $Q\succ 0$.{\color{black}

\begin{cor}\label{cor:QuadraticConditions}
Consider $\cA=\{A_1,\dots, A_M\}\subset \R^{n\times n}$, a $\rho>0$ and $\tau_2\geq \tau_1\geq 0$. Suppose there exist $P_1^+,P_1^-,\dots, P_M^+,P_M^-\succ 0$,  such that
\begin{subequations}\label{eq:ConditionsQuadratic}
\begin{align}
e^{A_i^\top t} P_i^+ e^{A_i t}&\preceq e^{-2\rho t}P_i^-,\,\forall i\in \cI,\,\forall t\in [0,\tau_2-\tau_1]\label{eq:Cond1Quadratic}\\
e^{A_i^\top\tau_1}P_j^-e^{A_i\tau_1}&\preceq e^{-2\rho\tau_1}P_i^+,\;\forall\,i\neq j\in \cI.\label{eq:COnd2Quad}
\end{align}
\end{subequations}
Then system~\eqref{eq:LinearSwitchedSystem} is $\rho$-UES on $\cS_{\dw}(\tau_1,\tau_2)$.
\end{cor}
\begin{proof}
The proof follows by Corollary~\ref{cor:LinearCase}, by defining $v_i^+(x):=\sqrt{x^\top P_i^+x}$, $v_i^-(x):=\sqrt{x^\top P_i^-x}$ for any $i\in \cI$.
\end{proof}
In Corollary~\ref{cor:QuadraticConditions} we have rewritten the conditions of Corollary~\ref{cor:LinearCase} in the case of \emph{quadratic norms}. On the other hand, these conditions still have the weakness of explicitly depending on the exponential matrices of the subsystems. Several possible relaxations/manipulations are possible in order to transforms~\eqref{eq:Cond1Quadratic}~\eqref{eq:COnd2Quad} into more treatable conditions. Among others, we mention the differential linear matrix inequality framework,~\cite[Chapter 2]{Geromel23}, or the discretizations techniques proposed in~\cite{Xiang15}. In what follows we illustrate a possible numerical scheme in order to verify the conditions of Corollary~\ref{cor:QuadraticConditions}.
}
{\color{black}
\begin{cor}\label{cor:LMIConditions}
Consider $\cA=\{A_1,\dots, A_M\}\subset \R^{n\times n}$, a $\rho>0$ and $\tau_2\geq \tau_1\geq 0$. Suppose there exist $P_1^+,P_1^-,\dots, P_M^+,P_M^-\succ 0$, $\mu\in (0,1) $ and $\nu\in \R$ such that the following inequalities
\begin{subequations}\label{eq:ConditionsLMI}
\begin{align}
P_i^+\preceq \mu^2 P_i^-,\;\;\;\;&\forall\,i\in \cI,\label{eq:LMI2}\\
A_i^\top P_i^+ + P_i^+ A_i\preceq 2\nu P_i^+,\;\;\;\;&\forall\,i\in \cI,\label{eq:LMI3}\\
\log(\mu)+(\tau_2-\tau_1)(\nu+\rho)\leq 0,\;\;\;\;\;\;\;\;&\label{eq:LMI4}
\end{align}
and condition~\eqref{eq:COnd2Quad} are satisfied.
\end{subequations}
Then system~\eqref{eq:LinearSwitchedSystem} is $\rho$-UES on $\cS_{\dw}(\tau_1,\tau_2)$.\\
Moreover, condition~\eqref{eq:COnd2Quad} can be replaced by the following statement: given a $K\in \N$, for every $i\neq j\in \cI$ there exist $Q_{ij,0},\dots, Q_{ij,K}\succ 0$ such that
\begin{equation}\label{eq:KRelaxation}
\begin{aligned}
Q_{ij,0}=P_i^+,&\;\;\;\;\;\;Q_{ij,K}=P_j^-,\\
A_i^\top Q_{ij,k}+Q_{ij,k}A_i&+\frac{K(Q_{ij,k}-Q_{ij,k-1})}{\tau_1}\prec 2\rho Q_{ij,k},\\
A_i^\top Q_{ij,k-1}+Q_{ij,k-1}A_i&+\frac{K(Q_{ij,k}-Q_{ij,k-1})}{\tau_1}\prec 2\rho Q_{ij,k-1}.
\end{aligned}
\end{equation}
\end{cor}

\begin{proof}
 Conditions~\eqref{eq:LMI2},~\eqref{eq:LMI3},~\eqref{eq:LMI4} imply condition~\eqref{eq:Cond1Quadratic} by~Lemma~\ref{lemma:AppendixAuxiliary2} in Appendix. The fact that~\eqref{eq:KRelaxation} implies~\eqref{eq:COnd2Quad} is recalled in Lemma~\ref{lemma:EasyImplicationLinear} in Appendix, see also~\cite[Theorem 2.5]{Geromel23}.
\end{proof}
}
{\color{black}
\begin{rem}[Robustness Issues] 
Condition~\eqref{eq:COnd2Quad} requires to compute the matrix exponentials $e^{A_i\tau_1}$ for any $i\in \cI$.  Instead, the conditions in~\eqref{eq:KRelaxation} (given an arbitrary $K\in \N$) are linear in $A_i$, $i\in \cI$, as~\eqref{eq:LMI3}, and they can thus be adapted also to the case where the matrices $A_1,\dots, A_N$ are affected by bounded noises/are (partially) unknown. Since in general the function $A \mapsto e^{A}$ is non-convex with respect to the components of $A$, this is not the case for~\eqref{eq:COnd2Quad}. Conditions~\eqref{eq:KRelaxation} require to fix a ``relaxation'' parameter $K\in \N$ and to increase the number of semidefinite decision variables and inequalities. We note that the equivalence between~\eqref{eq:COnd2Quad} and~\eqref{eq:KRelaxation} is reached only for arbitrarily large $K\in \N$, and thus this technique, once fixed a $K\in \N$ introduces, in general, conservatism to the proposed stability criteria. For more discussion on this topic, we refer to~\cite{AllSha11,Xiangxiao14,Xiang15} and the recent monograph~\cite{Geromel23}. 
\end{rem}}

\subsection{Numerical Examples}
\textcolor{black}{
We first modify a celebrated example taken from~\cite[Pag.19, Example~3.1]{Lib03}, and we present a switched linear system that exhibits stability for signals in the class of \emph{fixed time switching} of the form
\[
\cS_{\text{fix}}(\tau)=\cS_{\dw}(\tau,\tau)=\{\sigma \in \cS\;\vert\; t^\sigma_k-t^\sigma_{k-1}=\tau, \;\forall k\in \N \},
\] for a certain $\tau>0$. 
On the other hand, we show that the system is \emph{unstable} on the class of dwell-time signals $\overline \cS_{\dw}(\tau)$ defined in~\eqref{Eq:PuerlyDwellTime}. For a given perturbation parameter $\delta>0$ we show that the system is exponentially stable on the class $\cS_{\dw}(\tau,\tau+\delta)$, using the conditions of Corollary~\ref{cor:LMIConditions}.
This $\delta>0$ can be considered as an upper bound on the possible delays (in the switching schedule) that do not destabilize the system.}

\begin{example}{\emph{(Delay In Switching):}}\label{ex:Liberzon}
Consider
\begin{equation*}
 A_1:=\begin{pmatrix}-\varepsilon &-1\\   4&-\varepsilon\end{pmatrix}\;\;\;\text{and }\;A_2:=\begin{pmatrix}-\varepsilon &-4\\   1&-\varepsilon\end{pmatrix}, 
\end{equation*}
with $0<\varepsilon<1$ a fixed parameter. The matrices $A_1$ and $A_2$ are Hurwitz.
Consider the switched linear system
\begin{equation}\label{eq:LinearSwitchingExample}
\dot x=A_{\sigma(t)}x,
\end{equation}
where $\sigma:\R_+\to \{1,2\}$ lies in $\cS$. Roughly speaking, the trajectories of the subsystems $\dot x=A_1x$ and $\dot x=A_2x$ are elliptically converging spirals with vertical and horizontal major axis respectively, as sketched in Figure~\ref{fig:Liberzon}.
Let us call $\mathcal{E}=\{e_1,e_2\}$ the canonical basis, we define
\[
\overline{\tau}:=\min\{t>0\;\vert\; e_2^\top e^{A_1t}e_1=0\}.
\]
that is the time at which the solution $\Phi_1(t,e_1)$  reach again the line spanned by $e_1$. By linearity of $A_1$, this is the time that \emph{each solution} implies to reach again the line spanned by the initial condition, or equivalently, the time that each solution implies to span an angle of $\pi$ around the origin. It turns out that this time can be explicitly computed (by writing explicitly the exponential matrix $e^{A_1t}$), and we have $\overline \tau=\frac{\pi}{2}$, we avoid the computations here. Since $A_2$ is an orthogonal transformation of $A_1$ ( $A_2=Q^\top A_1Q$ with $Q=\begin{psmallmatrix}0 &1\\   -1&0\end{psmallmatrix}$)  the same hold for solutions of $\dot x=A_2 x$. 
Moreover, from a straightforward computation we have that \textcolor{black}{the matrix
$e^{A_2\overline \tau}e^{A_1\overline  \tau}$ is Schur-stable,
implying that the switched system~\eqref{eq:LinearSwitchingExample} is UES on $\cS_{\text{fix}}(\overline \tau)$.}
On the other hand, we note that~\eqref{eq:LinearSwitchingExample} is not GUAS on $\overline \cS_{\text{dw}}(\overline \tau)$. It suffices to choose 
$\tau_1:=\min\{t\geq\overline \tau\;\vert\;e_1^\top e^{A_1t}e_1=0\}$,
$\tau_2:=\min\{t\geq\overline \tau\;\vert\;e_2^\top e^{A_2t}e_2=0\}$
and the periodic (with period $\tau_1+\tau_2$) switching signal 
\begin{equation}\label{eq:DivergingSwitchingrule}
\sigma(t)=\begin{cases}
1,\;\;\text{if} \;\;t\in [0,\tau_1),\\
2,\;\;\text{if} \;\;t\in [\tau_1,\tau_1+\tau_2),
\end{cases}
\end{equation}
to show instability for $\varepsilon$ small enough, see Figure~\ref{fig:Liberzon} for an example of diverging trajectory.
\begin{figure}
\begin{center}
\includegraphics[scale=0.7]{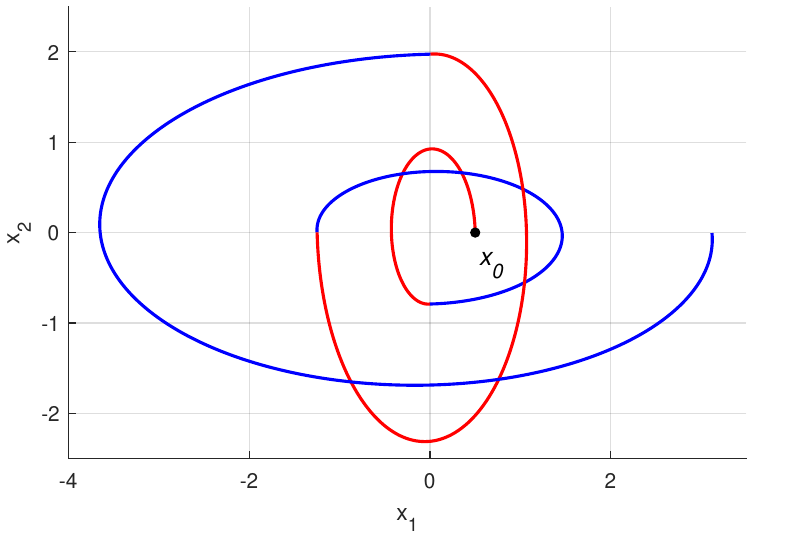}
\caption{Example of diverging trajectory of~\eqref{eq:LinearSwitchingExample} (with $\varepsilon=0.1$), starting at $x_0=[0.5,0]^\top$, under the switching rule defined in~\eqref{eq:DivergingSwitchingrule}. Red color stands for the subsystem $\dot x=A_1x$, blue for $\dot x=A_2x$.}  
\label{fig:Liberzon}
\end{center}
\end{figure}
In what follows, we fix $\varepsilon=0.1$, and we consider classes of the form $\cS_{\dw}(\overline \tau,\overline \tau+\delta)$ for  parameters $\delta>0$. Intuitively, we suppose that the switched system, which we proved UES on $\cS_{\text{fix}}(\overline \tau)$, now is possibly affected by unpredictable delays on the switching instants, and these delays are upper bounded by $\delta$. Using the conditions of Corollary~\ref{cor:LMIConditions}, we want to estimate an upper bound for the $\delta>0$ for which the stability of the system is ``preserved'', considering signals in $\cS_{\dw}(\overline \tau,\overline \tau+\delta)$. The destabilizing signal designed in~\eqref{eq:DivergingSwitchingrule} already tells us that this $\delta$ cannot be unbounded, since the system is unstable on $\overline \cS_{\dw}(\overline \tau)=\lim_{\delta\to +\infty}\cS_{\dw}(\overline \tau,\overline \tau+\delta)$. In  order to apply the conditions of~Corollary~\ref{cor:LMIConditions} we fix $\rho=0.001$, $\nu=1.5$ and we explicitly compute $e^{A_1\overline \tau}$ and $e^{A_2\overline \tau}$. Then,  we solve the LMIs
\[
\begin{aligned}
P_i^+&\preceq \mu^2 P_i^-,\;\;\;\;\forall\,i\in \{1,2\},\\
A_i^\top P_i^+ + P_i^+ A_i&\preceq 2\nu P_i^+,\;\;\;\;\forall\,i\in \{1,2\},\\
e^{A_i^\top\overline\tau}P_j^-e^{A_i\overline\tau}&\preceq e^{-2\rho\overline \tau}P_i^+,\;\;\;\;\forall\,i\neq j\in \{1,2\},
\end{aligned}
\]
minimizing, via line search, the parameter $\mu\in (0,1)$. The minimal value of $\mu$ for which the LMIs are feasible we were able to find is~$\widehat\mu=0.86$. Using condition~\eqref{eq:LMI4} in Corollary~\ref{cor:LMIConditions} we have that the system is $\rho$-UES on $\cS_{\dw}(\overline \tau,\overline \tau+\delta)$ if 
\[
\delta\leq-\frac{\ln(\widehat\mu)}{\nu+\rho}=-\frac{\ln(0.86)}{1.501}\approx 0.1.
\]
Thus, we have proven that for signals $\sigma\in \cS$ that switch every $\overline \tau$ units of time with a possible delay bounded from above by $\delta=0.1$, the exponential stability of~\eqref{eq:LinearSwitchingExample} is preserved.
\end{example}
We now borrow an example already considered in the literature, and we highlight how the conditions of Corollary~\ref{cor:LMIConditions} can provide less conservative results, in some cases.
\begin{example}{\emph{(Unstable subsystems)},~\cite[Section 5]{Xiangxiao14}}\\
Let us consider the matrices 
\begin{equation*}
 A_1:=\begin{pmatrix}-1.9 &0.6\\   0.6&-0.1\end{pmatrix}\;\;\;\text{and }\;A_2:=\begin{pmatrix}0.1 &-0.9\\   0.1&-1.4\end{pmatrix}, 
\end{equation*}
which are both Hurwitz unstable, i.e. they both have at least one eigenvalue with positive real part.
In~\cite[Section 5]{Xiangxiao14}  the system is studied, and using a different multiple quadratic Lyapunov function technique, values $0\leq\tau_1\leq \tau_2$ for which the corresponding switched linear system is stable on $\cS_{\dw}(\tau_1,\tau_2)$ are provided. It can be seen that, numerically, conditions of Corollary~\ref{cor:LMIConditions} are able to recover values consistent with the analysis performed in~\cite{Xiangxiao14}. Moreover, fixing $\mu=0.65$, $\nu=0.25$, $\rho=0.001$, we are able to check the feasibility of the conditions for $\tau_1=0.6$ (here considered as a parameter to minimize in solving~\eqref{eq:LMI2}~\eqref{eq:LMI3}~\eqref{eq:COnd2Quad}). This implies that the system is $\rho$-UES
on $\cS_{\dw}(\tau_1,\tau_2)$ with 
\[
\tau_2\leq -\frac{\log(\mu)}{\nu+\rho}+\tau_1\approx -\frac{0.4308}{0.251}+0.6\approx2.3163.
\]
These ``stabilizing'' values of $\tau_1$ and $\tau_2$ were not found with the techniques of~\cite{Xiangxiao14} (see Section 5 of the mentioned paper), which, we underline, are based on a splitting procedure similar to  the one used in~\eqref{eq:KRelaxation}.
\end{example} 

\section{Conclusion}\label{Sec:conclu}
In conclusion, this paper has provided a comprehensive exploration of the stability analysis of switched systems, considering signals with upper and lower \textcolor{black}{bounds on the distance between switching instants.} By adapting and extending the multiple Lyapunov functions approach, we provided a thorough characterization of uniform stability, both in the nonlinear and linear case. For the linear case, numerical schemes have been proposed in order to check the Lyapunov sufficient conditions.
As open route for future research, we expect to generalize the approach of this manuscript to broader classes of signals. Moreover, we expect to apply the current analysis to hybrid systems which exhibit upper and lower bounds on the time of occurrence of jump events.

\section*{Acknowledgements}
The author would like to thank Aneel Tanwani, Thiago Alves Lima and Lucas N. Egidio for the useful suggestions and the fruitful discussions.

\appendix
\section{Technical Lemmas}
In this Appendix we collect some technical statements in order to handle inequalities of the form~\eqref{eq:UpperLowerInequaility1}~\eqref{eq:UpperLowerInequaility2} (or, equivalently~\eqref{eq:UpperLowerInequaility221}~\eqref{eq:UpperLowerInequaility222}). In what follows, given a vector field $f:\R^n\to \R^n$ satisfying Assumption~\ref{assumt:Regularity}, $\Phi_f:\R_+\times \R^n\to \R^n$ denotes the corresponding solution/flow map.

 \begin{lemma}\label{lem:IntermadiateLemma}
Consider a vector field $f:\R^n\to \R^n$ satisfying the hypothesis of Assumption~\ref{assumt:Regularity}, and two continuous function $V_b,V_a:\R^n\to \R_+$ such that
\[
V_b(\Phi_f(\overline \tau,x))\leq e^{-\overline\tau}V_a(x),\;\;\forall x\in \R^n,
\]
for a given $\overline \tau>0$. Then, there exists a continuous function $U:[0,\overline\tau]\times \R^n\to \R$, such that
\begin{subequations}\label{eq:propertyFunctionW}
\begin{align}
&U(0,x)=V_a(x)\;\wedge\; U(\overline \tau,x)=V_b(x),\;\forall \,x\in \R^n,\label{eq:propertyFunctionW1}\\
&U(t,\Phi_f(t,x))\leq e^{-t}U(0,x),\;\forall\,(t,x)\in[0,\overline \tau]\times \R^n.\label{eq:propertyFunctionW3}
\end{align}
\end{subequations}
Moreover, if for some $\alpha_1,\alpha_2\in \cK_\infty$ it holds that $\alpha_1(|x|)\leq V_s(x)\leq \alpha_2(|x|)$ for all $x\in \R^n$ and $s\in \{a,b\}$, then there exist $\wt \alpha_1,\wt\alpha_2\in \cK_\infty$ such that \[
\wt \alpha_1(|x|)\leq \alpha_1(|x|)\leq U(t,x)\leq \alpha_2(|x|)\leq \wt \alpha_2(|x|),
\]
for all $(t,x)\in[0,\overline \tau]\times \R^n$.
 \end{lemma}
 \begin{proof}
The idea behind the proof is inspired by~\cite[Lemma 1]{DelPaq22}.
Let us define the auxiliary functions $Z_1,Z_2:[0,\overline \tau]\times \R^n\to\R$ by
\[
Z_1(t,x):=e^{-t}V_a(\Phi_f(-t,x)),\;\;Z_2(t,x):=V_b(\Phi_f({\color{black}\overline \tau}-t,x)).
\]
Let us consider any continuous and strictly increasing function $\rho:[0,\overline \tau]\to [0,1]$ such that $\rho(0)=0$ and $\rho(\overline \tau)=1$, and define 
\[
U(t,x):=\rho(t)Z_2(t,x)+(1-\rho(t))Z_1(t,x).
\]
It is easy to see that~\eqref{eq:propertyFunctionW1} holds; computing we have
\[
\begin{aligned}
U(t,\Phi_f(t,x))&=\rho(t)Z_2(t,\Phi_f(t,x))+(1-\rho(t))Z_1(t,\Phi_f(t,x))=\rho(t)V_b(\Phi_f(\overline \tau,x))+(1-\rho(t))(e^{-t}V_a(x))\\
&=e^{-t}V_a(x)+\rho(t)(V_b(\Phi_f(\overline\tau,x))-e^{-t}V_a(x))\leq e^{-t}V_a(x)+\rho(t)((e^{-\overline \tau}-e^{-t})V_a(x)))\\&\leq e^{-t}V_a(x).
\end{aligned}
\]
The last statement concerning the $\cK_\infty$ bounds follows by the continuous dependence of solutions of $\dot x=f(x)$ on initial conditions (see~\cite[Theorem 3.4]{khalil2002nonlinear}), we refer to~\cite[Lemma 1]{DelPaq22} for the complete argument.
 \end{proof}
We now present conditions that allow us to verify~\eqref{eq:UpperLowerInequaility1} (and, equivalently)~\eqref{eq:UpperLowerInequaility222} without computing explicitly the solutions of the subsystems.

\begin{lemma}\label{lemma:AppendixAuxiliary2}
Consider two locally Lipschitz functions $V_a,V_b:\R^n \to \R$ and a vector field $f:\R^n\to \R^n$ satisfying the hypothesis of Assumption~\ref{assumt:Regularity}. Let us fix $\rho>0$; if there exist $\mu\in (0,1)$ and $\nu\in \R$ such that
\[
\begin{aligned}
V_b(x)&\leq \mu V_a(x),\;\;\forall \;x\in \R^n;\\
D^+_f\,V_b(x)&\leq \nu V_b(x),\;\;\;\forall \;x\in \R^n;\\
\log(\mu)&+\overline \tau (\nu+\rho)\leq 0;
\end{aligned}
\]
then it holds that 
\[
V_b(\Phi_f(\tau,x))\leq e^{-\rho\tau}V_a(x),\;\;\forall x\in \R^n,\;\forall \tau\in [0,\overline \tau].
\]
\end{lemma}
\begin{proof}
For any $\tau \in [0,\overline \tau]$, using the comparison lemma (see for example~\cite[Lemma 3.4]{khalil2002nonlinear}), we have
\[
V_b(\Phi_f(\tau,x))\leq e^{\nu \tau} V_b(x)\leq \mu e^{\nu \tau} V_a(x).
\]
We suppose $\nu\geq -\rho$, otherwise the statement is trivial.
Given any $x\in \R^n$, we have
\[
\begin{aligned}
\mu e^{\nu \tau} V_a(x)\leq e^{-\rho\tau}V_a(x),\;\forall\,\tau\in [0,\overline \tau]\;\;\Leftrightarrow\;\;\mu \leq e^{-(\nu+\rho)\tau},\;\forall\,\tau\in [0,\overline \tau]\;\;\Leftrightarrow\; \;\log(\mu)+(\nu+\rho)\overline \tau\leq 0,
\end{aligned}
\]
concluding the proof.
\end{proof}
The following statements allows us to remove the explicit dependence on the subsystems solutions in the inequalities~\eqref{eq:UpperLowerInequaility2} and~\eqref{eq:UpperLowerInequaility221}.

\begin{lemma}\label{lemma:EasyImplication}
Consider $f:\R^n\to \R^n$ satisfying the hypothesis of Assumption~\ref{assumt:Regularity}, a $\overline \tau>0$ and a $K\in \N\setminus\{0\}$. Suppose there exist $V_0,\dots, V_K\in \cC^1(\R^n,\R)$ positive definite and  $\rho\in \R$ such that
\begin{equation}\label{eq:GradientSystems1}
\begin{cases}
\nabla V_k(x)\cdot f(x)\;\;\;\,+\frac{K(V_k(x)-V_{k-1}(x))}{\overline\tau}\leq \rho V_k(x),\;\forall\,x\in \R^n,\\
\nabla V_{k-1}(x)\cdot f(x)+\frac{K(V_k(x)-V_{k-1}(x))}{\overline\tau}\leq \rho V_{k-1}(x),\;\forall\,x\in \R^n.
\end{cases}
\end{equation}
for all $k\in \{1,\dots, K\}$. This implies that 
\begin{equation}\label{eq:decreasingProperty}
V_K(\phi_f(\overline \tau,x))\leq e^{\rho\overline \tau}V_0(x),\;\;\;\forall \;x\in \R^n.
\end{equation}
\end{lemma}
\begin{proof}[Sketch of the Proof]
The proof basically follows by~\cite[Lemma 2]{DelPaq22}.
Define $t_k=\frac{k}{K}\overline\tau$ for $k\in \{0,\dots, K\}$ and consider the function $U:[0,\overline\tau]\times \R^n\to \R$ as
\[
U(t,x):=\frac{K(t-t_{k-1})}{\tau}V_k(x)+\frac{K(t_k-t)}{\tau}V_{k-1}(x)\;\;
\]
for $t\in[t_{k-1},t_k]$ and for all $k\in\{1,\dots, K\}$. 
Using~\eqref{eq:GradientSystems1}, it can be seen that
\[
\begin{aligned}
\frac{d}{dt}U(t,\Phi_f(t,x))\leq \rho U(t,x),
\end{aligned}
\]
and then it suffices to apply the comparison lemma~\cite[Lemma 3.4]{khalil2002nonlinear}.
\end{proof}
We now provide, for completeness, the specification of Lemma~\ref{lemma:EasyImplication} in the case of linear subsystem (i.e. $f(x)\equiv Ax$, $\forall x\in \R^n$) and quadratic norms (i.e. $U_j(x)\equiv\sqrt{x^\top P_jx}$) for a certain $P_j\succ 0$). It turns out that, in this case, we have an equivalence.

\begin{lemma}\label{lemma:EasyImplicationLinear}
Consider $A\in \R^{n\times n}$, $\overline \tau>0$ and $\rho>0$. Given two $P_b,P_a\succ 0$ we have that 
\begin{equation}\label{eq:decreasingPropertyLinear}
e^{A^\top\overline \tau} P_be^{A\overline \tau}\prec e^{\rho 2\overline\tau}P_a,
\end{equation}
if and only if there exist $K\in \N$ and $P_0,\dots, P_K\succ 0$ positive definite such that $P_0=P_a$, $P_K=P_b$ and 
\begin{equation}\label{eq:GradientSystemsLinear1}
\begin{cases}
A^\top P_k+P_kA&+\frac{K(P_k-P_{k-1})}{\overline\tau}\prec 2\rho P_k,\\
A^\top P_{k-1}+P_{k-1}A&+\frac{K(P_k-P_{k-1})}{\overline\tau}\prec 2\rho P_{k-1}.
\end{cases}
\end{equation}
for all $k\in \{1,\dots, K\}$. 
\end{lemma}
For the proof, which basically follows the idea of proof of Lemma~\ref{lemma:EasyImplication}, we refer to~\cite{AllSha11,Xiangxiao14,Xiang15}. \textcolor{black}{For a recent general and self-contained discussion, we refer to~\cite[Section 2.4]{Geromel23}}.

\bibliography{biblio} 

\begin{thebibliography}{10}

\bibitem{ahmadi}
A.~A. Ahmadi, R.~M. Jungers, P.~A. Parrilo, and M.~Roozbehani.
\newblock Joint spectral radius and path-complete graph {Lyapunov} functions.
\newblock {\em SIAM Journal on Control and Optimization}, 52(1):687--717, 2014.

\bibitem{AllSha11}
L.~I. Allerhand and U.~Shaked.
\newblock Robust stability and stabilization of linear switched systems with
  dwell time.
\newblock {\em IEEE Transactions on Automatic Control}, 56(2):381--386, 2011.

\bibitem{AngeliNote}
D.~Angeli.
\newblock A note on stability of arbitrarily switched homogeneous systems.
\newblock Technical report, 1999.

\bibitem{BacRosier}
A.~Bacciotti and L.~Rosier.
\newblock {\em Liapunov Functions and Stability in Control Theory}, volume 267
  of {\em Lecture Notes in Control and Information Sciences}.
\newblock Springer-Verlag, 2005.

\bibitem{BlaCol10}
F.~{Blanchini} and P.~{Colaneri}.
\newblock Vertex/plane characterization of the dwell-time property for
  switching linear systems.
\newblock In {\em 49th IEEE Conference on Decision and Control (CDC)}, pages
  3258--3263, 2010.

\bibitem{Bri16}
C.~Briat.
\newblock Stability analysis and stabilization of stochastic linear impulsive,
  switched and sampled-data systems under dwell-time constraints.
\newblock {\em Automatica}, 74:279--287, 2016.

\bibitem{Briat17}
C.~Briat.
\newblock Dwell-time stability and stabilization conditions for linear positive
  impulsive and switched systems.
\newblock {\em Nonlinear Analysis: Hybrid Systems}, 24:198--226, 2017.

\bibitem{BriSeu13}
C.~Briat and A.~Seuret.
\newblock Affine characterizations of minimal and mode-dependent dwell-times
  for uncertain linear switched systems.
\newblock {\em IEEE Transactions on Automatic Control}, 58(5):1304--1310, 2013.

\bibitem{CheCol12}
G.~Chesi, P.~Colaneri, J.~C. Geromel, R.~Middleton, and R.~Shorten.
\newblock A nonconservative {LMI} condition for stability of switched systems
  with guaranteed dwell time.
\newblock {\em IEEE Transactions on Automatic Control}, 57(5):1297--1302, 2012.

\bibitem{ChiGugPro21}
Y.~Chitour, N.~Guglielmi, V.~Yu. Protasov, and M.~Sigalotti.
\newblock Switching systems with dwell time: Computing the maximal {Lyapunov}
  exponent.
\newblock {\em Nonlinear Analysis: Hybrid Systems}, 40:101021, 2021.

\bibitem{COLGer2008}
P.~Colaneri, J.C. Geromel, and A.~Astolfi.
\newblock Stabilization of continuous-time switched nonlinear systems.
\newblock {\em Systems \& Control Letters}, 57(1):95--103, 2008.

\bibitem{DebDelRos22}
V.~Debauche, M.~{Della Rossa}, and R.M. Jungers.
\newblock Comparison of path-complete {Lyapunov} functions via
  template-dependent lifts.
\newblock {\em Nonlinear Analysis: Hybrid Systems}, 46:101237, 2022.

\bibitem{DelPaq22}
M.~{Della Rossa}, M.~Pasquini, and D.~Angeli.
\newblock Continuous-time switched systems with switching frequency
  constraints: {Path}-complete stability criteria.
\newblock {\em Automatica}, 137:110099, 2022.

\bibitem{Geromel23}
J.~C. Geromel.
\newblock {\em Differential Linear Matrix Inequalities In Sampled-Data Systems
  Filtering and Control}.
\newblock Springer, 2023.

\bibitem{GerCol06}
J.~C. Geromel and P.~Colaneri.
\newblock Stability and stabilization of continuous-time switched linear
  systems.
\newblock {\em SIAM Journal on Control and Optimization}, 45(5):1915--1930,
  2006.

\bibitem{goebel2012hybrid}
R.~Goebel, R.G. Sanfelice, and A.R. Teel.
\newblock {\em Hybrid Dynamical Systems: Modeling, Stability, and Robustness}.
\newblock Princeton University Press, 2012.

\bibitem{HafTan23}
S.~Hafstein and A.~Tanwani.
\newblock Linear programming based lower bounds on average dwell-time via
  multiple {Lyapunov} functions.
\newblock {\em European Journal of Control}, page 100838, 2023.

\bibitem{HeemelsJohn12}
W.P.M.H. Heemels, K.H. Johansson, and P.~Tabuada.
\newblock An introduction to event-triggered and self-triggered control.
\newblock In {\em 2012 IEEE 51st IEEE Conference on Decision and Control
  (CDC)}, pages 3270--3285, 2012.

\bibitem{HesMor99}
J.P. Hespanha and A.S. Morse.
\newblock Stability of switched systems with average dwell-time.
\newblock In {\em Proceedings of the 38th IEEE Conference on Decision and
  Control}, volume~3, pages 2655--2660 vol.3, 1999.

\bibitem{Kellett2014}
C.M. Kellett.
\newblock A compendium of comparison function results.
\newblock {\em Mathematics of Control, Signals, and Systems}, 26(3):339--374,
  2014.

\bibitem{khalil2002nonlinear}
H.~K. Khalil.
\newblock {\em Nonlinear Systems}.
\newblock Pearson Education. Prentice Hall, 2002.

\bibitem{KunChat15}
A.~Kundu and D.~Chatterjee.
\newblock Stabilizing switching signals for switched systems.
\newblock {\em IEEE Transactions on Automatic Control}, 60(3):882--888, 2015.

\bibitem{Lib03}
D.~Liberzon.
\newblock {\em Switching in Systems and Control}.
\newblock Systems \& Control: Foundations \& Applications. Birkh{\"a}user,
  2003.

\bibitem{LinAnt09}
H.~{Lin} and P.~J. {Antsaklis}.
\newblock Stability and stabilizability of switched linear systems: A survey of
  recent results.
\newblock {\em IEEE Transactions on Automatic Control}, 54(2):308--322, 2009.

\bibitem{Mancilla00}
J.L. Mancilla-Aguilar and R.A. Garc\'ia.
\newblock A converse {Lyapunov} theorem for nonlinear switched systems.
\newblock {\em Systems \& Control Letters}, 41(1):67--71, 2000.

\bibitem{Morse}
A.S. Morse.
\newblock Supervisory control of families of linear set-point controllers -
  part i. exact matching.
\newblock {\em IEEE Transactions on Automatic Control}, 41(10):1413--1431,
  1996.

\bibitem{PhiEssDul16}
M.~Philippe, R.~Essick, G.E. Dullerud, and R.M. Jungers.
\newblock Stability of discrete-time switching systems with constrained
  switching sequences.
\newblock {\em Automatica}, 72:242--250, 2016.

\bibitem{ProKam23}
V.Yu. Protasov and R.~Kamalov.
\newblock Stability of continuous time linear systems with bounded switching
  intervals.
\newblock {\em SIAM Journal on Control and Optimization}, 61(5):3051--3075,
  2023.

\bibitem{ShoWir07}
R.~Shorten, F.~Wirth, O.~Mason, K.~Wulff, and C.~King.
\newblock Stability criteria for switched and hybrid systems.
\newblock {\em SIAM Review}, 49(4):545--592, 2007.

\bibitem{Son98}
E.D. Sontag.
\newblock Comments on integral variants of {ISS}.
\newblock {\em Systems \& Control Letters}, 34(1):93--100, 1998.

\bibitem{TeelPraly2000}
A.R. Teel and L.~Praly.
\newblock A smooth {Lyapunov} function from a class-$\mathcal{KL}$ estimate
  involving two positive semidefinite functions.
\newblock {\em ESAIM: COCV}, 5:313--367, 2000.

\bibitem{Wirth2005}
F.~Wirth.
\newblock A converse {Lyapunov} theorem for linear parameter-varying and linear
  switching systems.
\newblock {\em SIAM Journal on Control and Optimization}, 44(1):210--239, 2005.

\bibitem{WirthCDC05}
F.~Wirth.
\newblock A converse {Lyapunov} theorem for switched linear systems with dwell
  times.
\newblock In {\em Proceedings of the 44th IEEE Conference on Decision and
  Control}, pages 4572--4577, 2005.

\bibitem{Xiang15}
W.~Xiang.
\newblock On equivalence of two stability criteria for continuous-time switched
  systems with dwell time constraint.
\newblock {\em Automatica}, 54:36--40, 2015.

\bibitem{Xiangxiao14}
W.~Xiang and J.~Xiao.
\newblock Stabilization of switched continuous-time systems with all modes
  unstable via dwell time switching.
\newblock {\em Automatica}, 50(3):940--945, 2014.

\bibitem{YangJia14}
H.~Yang, B.~Jiang, and V.~Cocquempot.
\newblock A survey of results and perspectives on stabilization of switched
  nonlinear systems with unstable modes.
\newblock {\em Nonlinear Analysis: Hybrid Systems}, 13:45--60, 2014.

\bibitem{YangWang20}
W.~Yang, Y.-W. Wang, C.~Wen, and J.~Daafouz.
\newblock Exponential stability of singularly perturbed switched systems with
  all modes being unstable.
\newblock {\em Automatica}, 113:108800, 2020.

\bibitem{YinJay23}
H.~Yin, B.~Jayawardhana, and S.~Trenn.
\newblock Stability of switched systems with multiple equilibria: A mixed
  stable–unstable subsystem case.
\newblock {\em Systems \& Control Letters}, 180:105622, 2023.

\bibitem{ZhaoShi17}
X.~Zhao, P.~Shi, Y.~Yin, and S.K. Nguang.
\newblock New results on stability of slowly switched systems: A multiple
  discontinuous {Lyapunov} function approach.
\newblock {\em IEEE Transactions on Automatic Control}, 62(7):3502--3509, 2017.

\end{thebibliography}
 \bibliographystyle{plain}

\end{document}